\documentclass[english,11pt]{article}

\usepackage{amssymb,amsmath,amsfonts,amsthm,rotating,setspace,graphicx,float}
\usepackage{amssymb,amsmath,amsthm,rotating,setspace}
\usepackage{enumerate}

\usepackage[title]{appendix}

\usepackage{graphicx}

\usepackage{latexsym}

\usepackage{epsfig}

\usepackage{color}

\newtheorem{theorem}{Theorem}[section]
\newtheorem{lemma}[theorem]{Lemma}

\newtheorem{cor}[theorem]{Corollary}

\theoremstyle{definition}
\newtheorem{de}[theorem]{Definition}

\theoremstyle{definition}
\newtheorem{remark}[theorem]{Remark}

\theoremstyle{proposition}

\numberwithin{equation}{section}

\def\eq#1{(\ref{#1})}

\def\N{\mathbb{N}}

\def\R{\mathbb{R}}

\begin{document}

\title{Global and blow-up radial solutions for quasilinear elliptic systems arising in
the study of viscous, heat conducting fluids}

%\author{ }

\author{Marius Ghergu\footnote{School of Mathematics and Statistics,
    University College Dublin, Belfield, Dublin 4, Ireland; {\tt
      marius.ghergu@ucd.ie}}$\,$ \footnote{Institute of Mathematics Simion Stoilow of the Romanian Academy, 21 Calea Grivitei Street,
010702 Bucharest, Romania} $\,,\;\,$
Jacques Giacomoni \footnote{LMAP (UMR E2S-UPPA CNRS 5142) Universit\'e de Pau et des Pays de l'Adour, Avenue de l'Universit\'e, F-64013 Pau cedex, France; {\tt
      jacques.giacomoni@univ-pau.fr}}  $\;$ and $\;$
Gurpreet Singh\footnote{School of Mathematics,
    Trinity College Dublin, Dublin 2, Ireland; {\tt
      gurpreet.bajwa2506@gmail.com}}
}

%\date{}

%\subjclass[2010]{Primary 35B40; 35J75; Secondary 35J25; 35B51}

%\date{21 mai 2011}

\maketitle

\begin{abstract}
We study positive radial solutions of quasilinear elliptic systems with a gradient term in the form
$$
\left\{
\begin{aligned}
\Delta_{p} u&=v^{m}|\nabla u|^{\alpha}&&\quad\mbox{ in }\Omega,\\
\Delta_{p} v&=v^{\beta}|\nabla u|^{q} &&\quad\mbox{ in }\Omega,
\end{aligned}
\right.
$$
where $\Omega\subset\R^N$ $(N\geq 2)$ is either a ball or the whole space, $1<p<\infty$, $m, q>0$, $\alpha\geq 0$, $0\leq \beta\leq  m$ and $(p-1-\alpha)(p-1-\beta)-qm\neq  0$. We first classify all the positive radial solutions in case $\Omega$ is a ball, according to their behavior at the boundary. Then we obtain that the system has non-constant global solutions if and only if $0\leq \alpha<p-1$ and $mq<  (p-1-\alpha)(p-1-\beta)$. Finally, we describe the precise behavior at infinity for such positive global radial solutions by using properties of three component cooperative and irreducible dynamical systems.
\end{abstract}


\medskip


%\tableofcontents\begin{proof}

\section{Introduction and the main results}

In this paper we investigate positive radial solutions for quasilinear elliptic systems of the form
\begin{equation}\label{sys0}
\left\{
\begin{aligned}
\Delta_{p} u&=v^{m}|\nabla u|^{\alpha}&&\quad\mbox{ in }\Omega,\\
\Delta_{p} v&=v^{\beta}|\nabla u|^{q} &&\quad\mbox{ in }\Omega,
\end{aligned}
\right.
\end{equation}
where $1<p<\infty$, $\Delta_{p} u= {\rm div}(|\nabla u|^{p-2}\nabla u)$ stands for the standard $p$-Laplace operator  and
$\Omega\subset \R^N$ ($N\geq 2$) is either a ball $B_R$ centered at the origin and having radius $R>0$, or the whole space. The exponents in \eqref{sys0} satisfy
$$
1<p<\infty,\;\; \;\;m, \;q>0,\;\;\; \alpha\geq 0,\;\; \;  0\leq \beta\leq  m,
$$
and
\begin{equation}\label{deltaa}
\delta:= (p-1-\alpha)(p-1-\beta)-qm\neq  0.
\end{equation}

In the semilinear case $p=2$, $\alpha=\beta=0$, $m=1$, $q=2$, system
\eqref{sys0} was introduced by D\'iaz, Lazzo and Schmidt
\cite{DLS2005} as a prototype model in the study of dynamics of a
viscous, heat-conducting fluid. Considering a unidirectional flow,
independent of distance in the flow direction, the speed $u$ and the
temperature $\theta$ satisfy the coupled equations
\begin{equation}\label{diaz1}
\left\{
\begin{aligned}
u_t-\Delta u &=\theta&&\quad\mbox{ in }\Omega,\\
\theta_t-\Delta \theta &=|\nabla u|^2&&\quad\mbox{ in }\Omega.
\end{aligned}
\right.
\end{equation}
The source terms $\theta$ and $|\nabla u|^2$ represent the buoyancy
force and viscous heating, respectively. With the change of variable
$v=-\theta$, steady states of \eqref{diaz1} satisfy
\begin{equation}\label{diaz2}
\left\{
\begin{aligned}
\Delta u &=v&&\quad\mbox{ in }\Omega,\\
\Delta v &=|\nabla u|^2&&\quad\mbox{ in }\Omega,
\end{aligned}
\right.
\end{equation}
which is the semilinear version of \eqref{sys0} in the particular
case $p=2$, $\alpha=\beta=0$, $m=1$ and $q=2$. In \cite{DLS2005} was
obtained that system \eqref{diaz2} admits a positive solution which
blows up at the boundary of a ball; such a solution is also unique
for fixed data. Further, it was observed in \cite{DLS2005} that in
case of small dimensions $ N\leq 9$ there also exists a boundary
blow-up solution of \eqref{diaz2} that changes sign. The study in
\cite{DLS2005} was then carried over to time dependent systems in
\cite{DRS2007, DRS2008}. Recently Singh \cite{S2015}, Filippucci and
Vinti \cite{FV2017} extended the study of positive radial solutions
in \cite{DLS2005} to more general class of nonlinearities.

Recent results \cite{BFP2015, F2011, F2013} have discussed the
existence and nonexistence of positive solutions for systems of
inequalities of the above type in the frame of general quasilinear
differential operators. Quasilinear elliptic systems without
gradient terms have been extensively investigated in the last three
decades; see, e.g., the results in \cite{ACM2002, BV2000, BVG2010,
BVGr1999, BVP2001, CFMT2000}.

In this paper we study {\it non-constant positive radial solutions} of \eqref{sys0}, that is, solutions $(u,v)$ which fulfill:
\begin{itemize}
\item $u, v\in C^2(\Omega)$ are positive and radially symmetric;
\item $u$ and $v$ are not constant in any neighbourhood of the origin;
\item $u$ and $v$ satisfy \eqref{sys0}.
\end{itemize}
If $\Omega=\R^N$, solutions of \eqref{sys0} will be called {\it global solutions}.

The presence of the gradient terms $|\nabla u|^\alpha$ and $|\nabla
u|^q$  in the right-hand side entails a rich structure of the
solution set of \eqref{sys0} which we aim to investigate in the
following. Throughout this work, we identify radial solutions
$(u,v)$ by their one variable representant, that is, $u(x)=u(r)$,
$v(x)=v(r)$, $r=|x|$. In the following, for a function $f:(0,R)\to
\R$ we denote $f(R^-)=\lim_{r\nearrow R} f(r)$, provided such a
limit exists.

In our first result below we classify all non-constant positive radial solutions in a ball $B_R$ according to
their behavior at the boundary. We have:

\begin{theorem}\label{thm1}
Assume $\Omega=B_R$, $1<p<\infty$, $m,q>0$, $0\leq \alpha<p-1$, $0\leq \beta\leq m$ and $\delta\neq 0$. Then
\begin{enumerate}
\item [(i)] There are no positive radial solutions $(u, v)$ with $u(R^{-})=\infty$ and $v(R^{-})<\infty$.
\item [(ii)] All positive radial solutions of \eq{sys0} are bounded if and only if
$$
mq< (p-1-\alpha)(p-1-\beta).
$$
\item [(iii)] There are positive radial solutions $(u, v)$ of \eq{sys0} with $u(R^{-})<\infty$ and $v(R^{-})=\infty$
if and only if
$$
mq> mp+(p-\alpha)(p-1-\beta).
$$
\item [(iv)] There are positive radial solutions $(u, v)$ of \eq{sys0} with $u(R^{-})= v(R^{-})=\infty$ if and only if
$$
(p-1-\alpha)(p-1-\beta)< mq\leq mp+(p-\alpha)(p-1-\beta).
$$
\end{enumerate}

\end{theorem}

Our next result concerns the existence of non-constant global positive radial solutions of \eqref{sys0}.
We obtain the following optimal result:

\begin{theorem}\label{thm2}
Assume $\Omega=\R^N$, $p>1$, $m,q>0$, $\alpha\geq 0$, $0\leq \beta\leq m$ and $\delta\neq 0$.
Then, \eqref{sys0} admits non-constant global positive radial solutions if and only if
\begin{equation}\label{nonct}
0\leq \alpha<p-1\quad\mbox{ and }\quad mq< (p-1-\alpha)(p-1-\beta).
\end{equation}
\end{theorem}
We next discuss the behavior at infinity of global positive radial solutions of \eqref{sys0}. Using properties
of three-component irreducible dynamical systems we are able to extend the result in \cite[Theorem 2.7]{S2015}
where extra conditions on exponents are required. For the sake of completeness, we have stated in Appendix A
all the important results from the theory of cooperative and irreducible dynamical systems we used in the present work.

\begin{theorem}\label{thm3}
Assume that $0\leq \alpha<p-1$ and $\delta>0$. Then, any
non-constant positive radial solution $(u, v)$ of \eq{sys0}
satisfies
\begin{equation}\label{abconst}
\lim_{|x|\rightarrow \infty} \frac{u(x)}{|x|^{1+\frac{p(m+1)-(1+\beta)}{\delta}}}= A \mbox{ and }
\lim_{|x|\rightarrow \infty} \frac{v(x)}{|x|^{\frac{p(p-1-\alpha)+q}{\delta}}}= B,
\end{equation}
where $A= A(N, p, q, m, \alpha, \beta)>0$ and $B= B(N, p, q, m, \alpha, \beta)>0$ have explicit
expressions  given by \eq{th2i} and \eq{th2j}.
\end{theorem}
The quantities
$A|x|^{1+\frac{p(m+1)-(1+\beta)}{\delta}}$ and
$B|x|^{\frac{p(p-1-\alpha)+q}{\delta}}$ that appear in Theorem
\ref{thm3} may be regarded as stabilizing profiles for the
steady-states solutions in the time-depending system that
corresponds to \eqref{sys0}.

We point out that the requirement $\delta>0$ in \eqref{deltaa} is a classical condition on superlinearity
of the system as described in \cite{BVG2010}. Also, the value of the limits $A$ and $B$ in \eqref{abconst}
depend decreasingly on the space dimension $N\geq 2$. One can see that from their expressions in \eq{th2i} and \eq{th2j}.

Using MATLAB we have plotted  the non-constant positive global
solution $(u,v)$ to \eqref{sys0} (see Figure 1 below) over the
interval $[0,500]$ for $p=10$, $\alpha=\beta=1$, $m=2$, $q=4$ and
for various space dimensions $N=3, 10, 30, 60$. The solutions was
normalized at the origin by $u(0)=v(0)=1$.

\begin{figure}[!htb]
   \begin{minipage}{0.53\textwidth}
     \centering
     \includegraphics[width=0.95\linewidth]{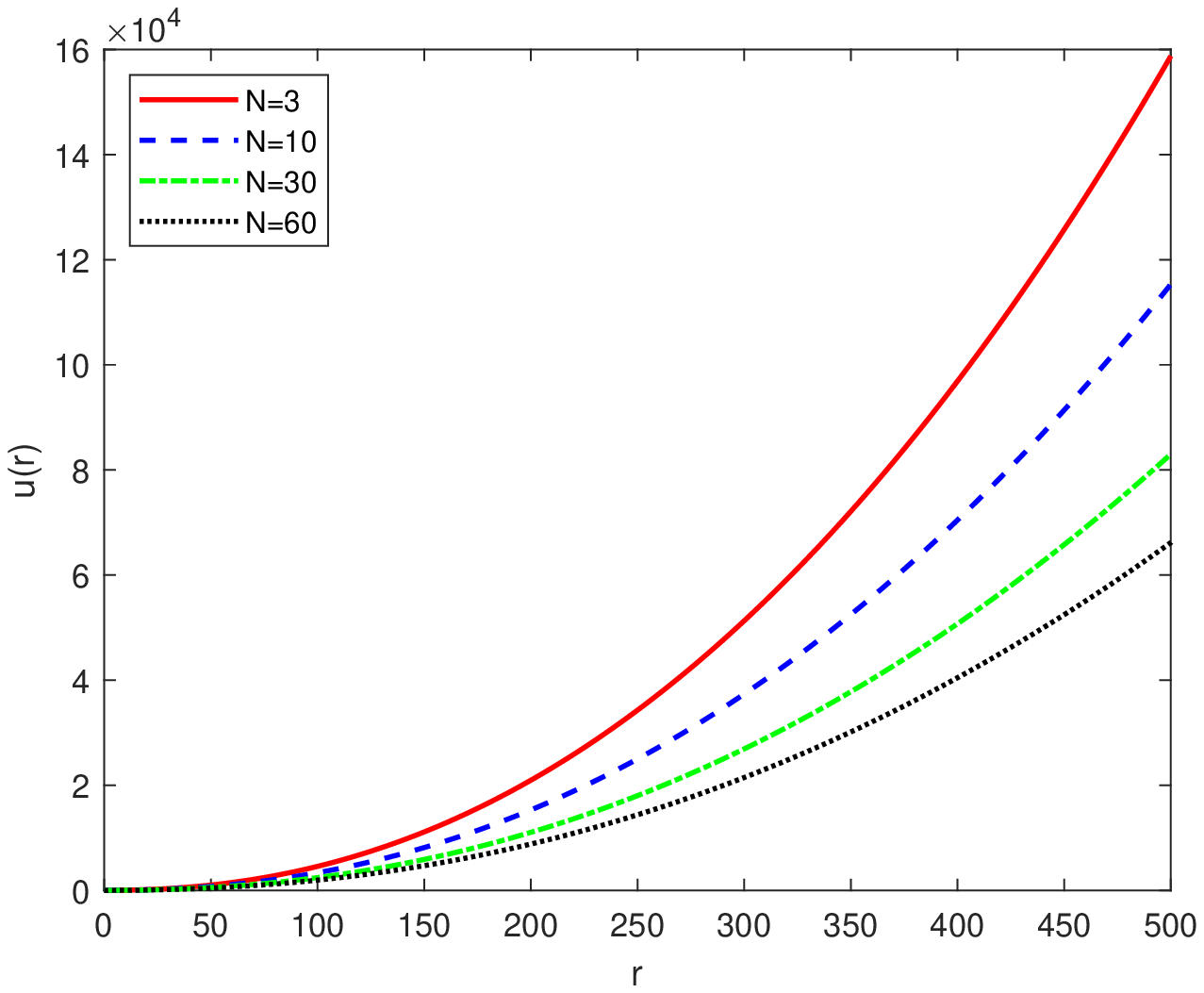}
     %\caption{Interpolation for Data 1}\label{Fig:Data1}
   \end{minipage}\hfill
   \begin {minipage}{0.53\textwidth}
     \centering
     \includegraphics[width=0.95\linewidth]{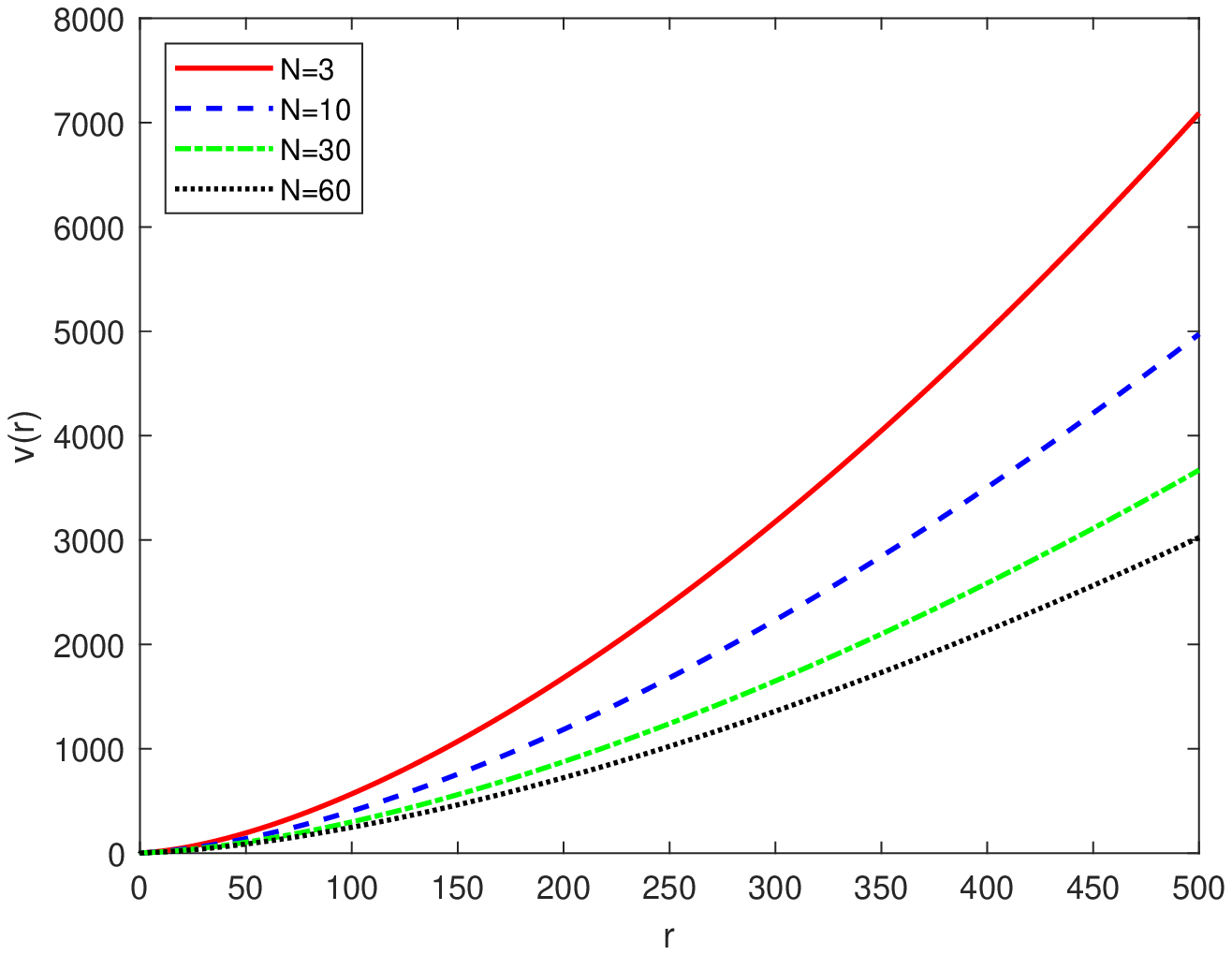}
     %\caption{Interpolation for Data 2}\label{Fig:Data2}
   \end{minipage}
   \caption{The graphs of $u(r)$ (left) and $v(r)$ (right) for $p=10$, $\alpha=\beta=1$, $m=2$,
   $q=4$ and for various space dimensions $N=3, 10, 30, 60$.}
\end{figure}

In our next result we show that given any pair $(a,b)\in (0,\infty)\times (0,\infty)$, there exists a
unique positive global radial solutions of \eqref{sys0} that emanates from $(a,b)$.

\begin{theorem}\label{thm4}
Assume $1<p<N$, $0\leq \alpha<p-1$ and $\delta> 0$. Then for any $a> 0$, $b> 0$ there exists a unique
global positive radial solution of \eq{sys0} such that $u(0)= a$ and $v(0)= b$.
\end{theorem}
Finally, let us consider the single equation that underlays the system \eqref{sys0}, namely
\begin{equation}\label{eqq}
\Delta_p u=u^m|\nabla u|^q\quad\mbox{ in }\R^N, N\geq 2.
\end{equation}
The case $q=0$ was discussed in \cite{MP2001}. Here we are interested in the case $m,q>0$. From
Theorems \ref{thm1}-\ref{thm4} we find:
\begin{cor}\label{coro}
Assume $m,q>0$, $p>1$ and $m+q\neq p-1$.
Then \eqref{eqq} has a non-constant positive radial solution if and only if
\begin{equation}\label{conn}
0<q<p-1\quad\mbox{ and }\quad m<p-q-1.
\end{equation}
If \eqref{conn} holds, then any non-constant positive radial solution $u$ of \eqref{eqq} satisfies
$$
\lim_{|x|\to \infty}\frac{u(x)}{|x|^{\frac{p-q}{p-1-m-q}}}=C(N,m,p,q)>0.
$$
Further, if $1<p<N$ then from any $a>0$ there exists a unique non-constant positive radial solution
$u$ of \eqref{conn} such that $u(0)=a$.
\end{cor}

The remaining of the paper contains the proofs of the above four theorems.

\section{Proof of Theorem \ref{thm1}}
Let $(u, v)$ be a non-constant positive radial solution of \eq{sys0} in a ball $B_R$. Then $(u, v)$ satisfies
\begin{equation}\label{eqthm2}
\left\{
\begin{aligned}
&\Big[ r^{N-1}u'|u'|^{p-2}\Big]'=r^{N-1}v^{m}|u'|^{\alpha} &&\quad\mbox{ for all } 0<r<R,\\
&\Big[ r^{N-1}v'|v'|^{p-2}\Big]'=r^{N-1}v^{\beta}|u'|^{q} &&\quad\mbox{ for all } 0<r<R,\\
&u'(0)=v'(0)=0, u(r)>0, v(r)>0 &&\quad\mbox{ for all } 0<r<R.
\end{aligned}
\right.
\end{equation}
Thus, $r\longmapsto r^{N-1}u'|u'|^{p-2}$ and $r\longmapsto r^{N-1}v'|v'|^{p-2}$ are nondecreasing and vanish at $r=0$.
Since $(u,v)$ is non-constant, it follows that $u'(r)> 0$ and $v'(r)> 0$ for all $0<r<R$, so $u$ and $v$ are increasing.
Thus, \eqref{eqthm2} reads

\begin{equation}\label{p1}
\left\{
\begin{aligned}
&\big[(u')^{p-1}\big]'+\frac{N-1}{r}(u')^{p-1}=v^{m}(u')^{\alpha} &&\quad\mbox{ for all } 0<r<R,\\
&\big[(v')^{p-1}\big]'+\frac{N-1}{r}(v')^{p-1}=v^{\beta}(u')^{q} &&\quad\mbox{ for all } 0<r<R,\\
&u'(0)=v'(0)=0, u(r)>0, v(r)>0 &&\quad\mbox{ for all } 0<r<R,
\end{aligned}
\right.
\end{equation}
which further implies
\begin{equation}\label{p1.1}
\left\{
\begin{aligned}
&\big[(u')^{p-1-\alpha}\big]'+\frac{\gamma}{r}(u')^{p-1-\alpha}=\frac{p-1-\alpha}{p-1} v^{m} &&\quad\mbox{ for all } 0<r<R,\\
&\big[(v')^{p-1}\big]'+\frac{N-1}{r}(v')^{p-1}=v^{\beta}(u')^{q} &&\quad\mbox{ for all } 0<r<R,\\
&u'(0)=v'(0)=0, u(r)>0, v(r)>0 &&\quad\mbox{ for all } 0<r<R,
\end{aligned}
\right.
\end{equation}
where
\begin{equation}\label{gama}
\gamma=\frac{(N-1)(p-1-\alpha)}{p-1}>0.
\end{equation}
We can rearrange \eqref{p1.1} in the form
\begin{equation}\label{p2}
\left\{
\begin{aligned}
&\Big[r^{\gamma}(u')^{p-1-\alpha}\Big]'=\frac{p-1-\alpha}{p-1}r^{\gamma}v^{m}(r) &&\quad\mbox{ for all } 0<r<R,\\
&\Big[r^{N-1}(v')^{N-1}\Big]'=r^{N-1}v^{\beta}(r)|u'(r)|^{q}, &&\quad\mbox{ for all } 0<r<R.
\end{aligned}
\right.
\end{equation}

\medskip

\begin{lemma}\label{firstestim}
Any non-constant positive radial solution $(u, v)$ of \eqref{sys0}
in $B_R$ satisfies
\begin{equation}\label{l01}
\Big(N+\frac{\alpha}{p-1-\alpha}\Big) (u'(r))^{p-1-\alpha} <r v^m(r) \quad\mbox{ for all }0<r<R,
\end{equation}

\begin{equation}\label{l02}
 N(v'(r))^{p-1} <r v^\beta(r) (u'(r))^q \quad\mbox{ for all }0<r<R,
\end{equation}

\begin{equation}\label{l1}
\frac{p-1-\alpha}{N(p-1-\alpha)+\alpha} v^{m}(r) < \big[(u')^{p-1-\alpha}\big]'(r)< \frac{p-1-\alpha}{p-1}v^{m}(r)\quad\mbox{ for all }0<r<R,
\end{equation}
and
\begin{equation}\label{l2}
\frac{v^{\beta}(r)(u')^{q}(r)}{N}\leq [(v')^{p-1}]'(r)\leq v^{\beta}(r)(u')^{q}(r) \quad\mbox{ for all }0<r<R.
\end{equation}

\end{lemma}

\begin{proof} For simplicity, let us write $w=u'$ so \eq{p1.1} and \eqref{p2} read

\begin{equation}\label{p3}
\left\{
\begin{aligned}
&(w^{p-1-\alpha})'+\frac{\gamma}{r}w^{p-1-\alpha}=\frac{p-1-\alpha}{p-1}v^m  &&\quad\mbox{for all }0<r<R, \\
&[(v')^{p-1}]'+\frac{N-1}{r}(v')^{p-1}=w^{q}v^{\beta}&&\quad\mbox{for all }0<r<R,\\
&w(0)=v'(0)=0, w(r)>0, v'(r)>0, v(r)>0 &&\quad\mbox{for all }0<r<R.
\end{aligned}
\right.
\end{equation}
and
\begin{equation}\label{p3.1}
\left\{
\begin{aligned}
&\Big[r^{\gamma}w^{p-1-\alpha}\Big]'=\frac{p-1-\alpha}{p-1}r^{\gamma}v^{m}(r)>0 &&\quad\mbox{ for all } 0<r<R,\\
&\Big[r^{N-1}(v')^{p-1}\Big]'=r^{N-1}v^{\beta}(r) w^{q}(r), &&\quad\mbox{ for all } 0<r<R, \\
&w(0)=v'(0)=0, w(r)>0, v'(r)>0, v(r)>0 &&\quad\mbox{for all }0<r<R.
\end{aligned}
\right.
\end{equation}
Integrating the first equation of \eqref{p3.1} and using the fact that $v$ is strictly increasing on $(0, R)$ we deduce
\begin{equation}\label{estimates2}
\begin{aligned}
r^{\gamma}w^{p-1-\alpha}(r)&=\frac{p-1-\alpha}{p-1}\int_0^r t^\gamma
v^{m}(t)dt
\\
&<\frac{p-1-\alpha}{p-1} v^{m}(r) \int_0^r  t^\gamma  dt\\
&=\frac{p-1-\alpha}{(p-1)(\gamma+1)}r^{\gamma+1} v^{m}(r)
\quad\mbox{ for all }0<r<R.
\end{aligned}
\end{equation}
Hence,
\begin{equation*}
w^{p-1-\alpha}(r) <\frac{p-1-\alpha}{(p-1)(\gamma+1)}r v^{m}(r)
\quad\mbox{ for all }0<r<R,
\end{equation*}
which proves \eqref{l01}.
Using this estimate in the first equation of \eq{p3} it follows that
\begin{equation*}
\begin{aligned}
\frac{p-1-\alpha}{p-1} v^m(r)&=(w^{p-1-\alpha})'(r)+\frac{\gamma}{r} w^{p-1-\alpha}(r)\\
&<(w^{p-1-\alpha})'(r)+ \frac{\gamma(p-1-\alpha)}{(\gamma+1)(p-1)}v^{m}(r) \quad\mbox{ for all } 0<r<R,
\end{aligned}
\end{equation*}
which implies
\begin{equation}\label{half1}
(w^{p-1-\alpha})'(r)>\frac{p-1-\alpha}{(\gamma+1)(p-1)}v^{m}(r)=\frac{p-1-\alpha}{N(p-1-\alpha)+\alpha}v^{m}(r) \quad\mbox{ for all } 0<r<R.
\end{equation}
Also, from \eqref{p3.1} and the positivity of $w$ we deduce
\begin{equation}\label{half2}
(w^{p-1-\alpha})'(r)<\frac{p-1-\alpha}{p-1}v^{m}(r) \quad\mbox{ for all } 0<r<R.
\end{equation}
Now, \eqref{l1} follows from \eqref{half1} and \eqref{half2}. At this point, let us note that from \eqref{l1} we have
that $(u')^{p-1-\alpha}$ is positive and strictly increasing so $w=u'$ is also positive and strictly increasing.
Using this fact and the same approach as above we derive \eqref{l02} and \eqref{l2}.
\end{proof}

\noindent{\bf Proof of Theorem \ref{thm1}.} The existence of a non-constant positive solution to \eqref{sys0} in a
small ball $B_\rho$ follows from similar arguments to \cite[Proposition A1]{FV2017}
(see also \cite[Proposition 9]{BFP2015}). Specifically, we employ a fixed point argument for the mapping
\begin{equation}\label{t1}
{\mathcal T}:C^1[0,\rho]\times C^1[0,\rho]\to C^1[0,\rho]\times C^1[0,\rho],
\end{equation}
given by
\begin{equation}\label{t2}
{\mathcal T}[u,v](r)=\left[\begin{array}{c}{\mathcal
T_1}[u,v](r)\\{\mathcal T_2}[u,v](r)\end{array}\right],
\end{equation}
where
\begin{equation}\label{t3}
\left\{
\begin{aligned}
&&{\mathcal T_1}[u,v](r)&=a+\int_0^r\left(\frac{p-1-\alpha}{p-1}t^{-\gamma}\int_0^t s^\gamma v^m(s)ds\right)^{1/(p-1-\alpha)}dt, \\
&&{\mathcal T_2}[u,v](r)&=b+\int_0^r\left(t^{1-N}\int_0^t s^{N-1}
v^\beta(s) |u'(s)|^qds\right)^{1/(p-1)}dt,
\end{aligned}
\right.
\end{equation}
where $a,b>0$. With a standard approach, there exists a small radius
$\rho>0$ such that ${\mathcal T}$ has a fixed point $(u,v)$ which is
a non-constant positive radially symmetric solution of
\eqref{eqthm2}. Now, the pair $(u_\lambda,v_\lambda)$ defined as
$$
u_\lambda(x)=\lambda^{1+\frac{p(m+1)-(1+\beta)}{\delta}}u\Big(\frac{x}{\lambda}\Big)\,,\quad v_\lambda(x)=\lambda^{\frac{p(p-1-\alpha)+q}{\delta}}v\Big(\frac{x}{\lambda}\Big),
$$
provides a non-constant positive radially symmetric solution of \eqref{eqthm2} in the
ball $B_{\lambda\rho}$. This shows that in any ball of positive radius there are non-constant positive
radially symmetric solution of \eqref{eqthm2}.

Let us assume now that $(u,v)$ is a non-constant positive radially symmetric solution of \eqref{eqthm2}
in $\Omega=B_R$, $R>0$ and set $z=(u')^{p-1-\alpha}$. Then  by \eq{l1} and \eq{l2} in Lemma \ref{firstestim} we have
\begin{equation}\label{p4}
Cv^{m}(r)\leq z'(r)\leq v^{m}(r) \quad\mbox{ for all }0<r<R,
\end{equation}
and
\begin{equation}\label{p5}
C v^{\beta}(r)z^{\frac{q}{p-1-\alpha}}(r)\leq [(v')^{p-1}]'(r)\leq v^{\beta}(r)z^{\frac{q}{p-1-\alpha}}(r) \quad\mbox{ for all }0<r<R,
\end{equation}
for some constant $C=C(N,p,\alpha)\in (0,1)$.
\begin{enumerate}
\item [(i)] Assume that $u(R^{-})= \infty$ and $v(R^{-})< \infty$. Since, $u'$ is increasing
(observe from \eq{p4} that $z$ is increasing, which implies $u'$ is also increasing) we deduce that
$u'(R^{-})= z(R^{-})= \infty$. Also, from \eq{p4} we find

\begin{equation*}
C_1< z'(r)(r)\leq C_2 \mbox{ for all } r\in (0, R),
\end{equation*}
for some positive constants $C_1,C_2$. Integrating over $[0, R]$  we
reach a contradiction.

\item [(ii)-(iv)]

\medskip

Let $(u, v)$ be a positive radial solution of \eq{sys0} with $v(R^{-})= \infty$. It follows from \eq{p4} that
$z'(R^{-})= \infty$. Also, $v'$ is increasing and $v(R^{-})= \infty$ imply $v'(R^{-})= \infty$.
Using \eq{p4} and \eq{p5} we have

\begin{equation}\label{p6}
z'(r)\leq v^{m}(r) \quad\mbox{ for all }0<r<R,
\end{equation}
and
\begin{equation}\label{p7}
C v^{\beta}(r)z^{\frac{q}{p-1-\alpha}}(r)\leq [(v')^{p-1}]'(r) \quad\mbox{ for all }0<r<R.
\end{equation}

Multiplying \eq{p6} and \eq{p7} we obtain
\begin{equation*}
z^{\frac{q}{p-1-\alpha}}(r)z'(r)\leq Cv^{m-\beta}(r)[(v')^{p-1}]'(r) \quad\mbox{ for all }0<r<R.
\end{equation*}
Integrating over $[0, R]$ in the above estimate we have
\begin{equation*}
z^{\frac{q+p-1-\alpha}{p-1-\alpha}}(r)\leq C \int_0^r v^{m-\beta}(t)[(v')^{p-1}]'(t)dt \leq Cv^{m-\beta}(r)\int_0^r [(v')^{p-1}]'(t)dt.
\end{equation*}
So,
\begin{equation}\label{p8}
z^{\frac{q+p-1-\alpha}{(p-1)(p-1-\alpha)}}(r)\leq Cv^{\frac{m-\beta}{p-1}}(r)v'(r) \quad\mbox{ for all }0<r<R.
\end{equation}
Multiplying \eq{p8} by $z'(r)$ and using \eq{p6} we have
\begin{equation*}
z^{\frac{q+p-1-\alpha}{(p-1)(p-1-\alpha)}}(r)z'(r)\leq Cv^{\frac{m-\beta}{p-1}+m}(r)v'(r) \quad\mbox{ for all }0<r<R,
\end{equation*}
that is
\begin{equation*}
z^{\frac{q+p-1-\alpha}{(p-1)(p-1-\alpha)}}(r)z'(r)\leq Cv^{\frac{mp-\beta}{p-1}}(r)v'(r) \quad\mbox{ for all }0<r<R.
\end{equation*}
A further integration over $[0, r]$, $0<r<R$, yields
\begin{equation*}
z^{\frac{q+p(p-1-\alpha)}{(p-1)(p-1-\alpha)}}(r)\leq Cv^{\frac{mp+p-1-\beta}{p-1}}(r) \quad\mbox{ for all }0<r<R.
\end{equation*}
Hence, by the first estimate in \eqref{p4} we find
\begin{equation*}
\begin{aligned}
z^{\frac{q+p(p-1-\alpha)}{p-1-\alpha}}(r)&\leq C v^{mp+p-1-\beta}(r)=
\big( v^m(r)\big) ^{\frac{mp+p-1-\beta}{m}}\\
&\leq C \big(z'(r)\big)^{\frac{mp+p-1-\beta}{m}}  \quad\mbox{ for all } 0<r<R,
\end{aligned}
\end{equation*}
which yields
\begin{equation}\label{p9}
z'(r)z^{-\sigma}(r)\geq C \quad\mbox{ for all } 0<r<R,
\end{equation}
where
\begin{equation}\label{p9a}
\sigma= \frac{m}{p-1-\alpha}\cdot \frac{q+p(p-1-\alpha)}{mp+p-1-\beta}
>0.
\end{equation}

Now, we return to \eq{p5} to get
\begin{equation*}
[(v')^{p-1}]'(r)\leq v^{\beta}(r)z^{\frac{q}{p-1-\alpha}}(r) \quad\mbox{ for all }0<r<R.
\end{equation*}
Multiplying by $v'(r)$ and integrating over $[0, r]$, we have
\begin{equation*}
\frac{(p-1)}{p}(v')^{p}(r)\leq \int_0^r v^{\beta}(t)v'(t)z^{\frac{q}{p-1-\alpha}}(t)dt \leq z^{\frac{q}{p-1-\alpha}}(r)\int_0^r v^{\beta}(t)v'(t)dt.
\end{equation*}
Hence
\begin{equation}\label{p10}
v'(r)v^{-\frac{\beta+1}{p}}(r)\leq Cz^{\frac{q}{p(p-1-\alpha)}}(r) \quad\mbox{ for all }0<r<R.
\end{equation}
Multiplying \eq{p10} by $z'(r)$ and using $z'(r)\geq Cv^{m}(r)$ we
find
\begin{equation*}
v'(r)v^{m-\frac{\beta+1}{p}}(r)\leq Cz^{\frac{q}{p(p-1-\alpha)}}(r)z'(r)\quad\mbox{ for all }0<r<R.
\end{equation*}
Further integration over $[0, r]$ yields
\begin{equation}\label{p11}
v^{\frac{mp+p-\beta-1}{p}}(r)-v^{\frac{mp+p-\beta-1}{p}}(0)\leq Cz^{\frac{q+p(p-1-\alpha)}{p(p-1-\alpha)}}(r) \quad\mbox{ for all }0<r<R.
\end{equation}
Since $v(R^{-})=\infty$, there exists $\rho \in (0,R)$ such that
\begin{equation}\label{p12}
v(r)\leq Cz^{\frac{q+p(p-1-\alpha)}{(p-1-\alpha)(mp+p-\beta-1)}}(r) \quad\mbox{ for all }\rho \leq r<R.
\end{equation}
This yields $z(R^{-})=\infty$. Then, using \eq{p4} and \eq{p12} we obtain
\begin{equation}\label{p13}
z'(r)\leq v^{m}(r) \leq Cz^{\sigma}(r) \quad\mbox{ for all }\rho \leq r<R,
\end{equation}
where the exponent $\sigma$ is defined in \eq{p9a}. It follows from \eq{p9} and \eq{p13} that
\begin{equation}\label{p14}
C_1\leq z'(r)z^{-\sigma}(r)\leq C_2 \quad\mbox{ for all }\rho \leq r<R,
\end{equation}
for some constants $C_2>C_1>0$ depending only on parameters $p,q,\alpha,\beta$ and dimension $N$.
Since $z(R^{-})= \infty$, we deduce from \eq{p14} that
\begin{equation}\label{p15}
\sigma> 1
\end{equation}
and
\begin{equation}\label{p16}
C_1(R-r)\leq \frac{z^{1-\sigma}(r)}{\sigma-1}\leq C_2(R-r) \quad\mbox{ for all }\rho \leq r<R.
\end{equation}
Using $z= (u')^{p-1-\alpha}$, we have
\begin{equation*}
C_1(R-r)^{-\frac{1}{(\sigma-1)(p-1-\alpha)}}\leq u'(r)\leq C_2(R-r)^{-\frac{1}{(\sigma-1)(p-1-\alpha)}} \quad\mbox{ for all }\rho \leq r<R.
\end{equation*}
Since
$$
u(R^{-})= u(\rho)+\int_\rho^R u'(t)dt,
$$
we deduce that
\begin{equation}\label{p17}
\begin{aligned}
u(R^{-})< \infty &\Longleftrightarrow \int_\rho^R (R-t)^{-\frac{1}{(\sigma-1)(p-1-\alpha)}}dt\\
& \Longleftrightarrow \int_0^1 s^{-\frac{1}{(\sigma-1)(p-1-\alpha)}}ds< 0\\
& \Longleftrightarrow \sigma> \frac{p-\alpha}{p-1-\alpha},
\end{aligned}
\end{equation}
and similarly
\begin{equation}\label{p18}
u(R^{-})= \infty \Longleftrightarrow \sigma \leq \frac{p-\alpha}{p-1-\alpha},
\end{equation}
From \eq{p16}, \eq{p17} and \eq{p18} we deduce that:

There are solutions $u(R^{-})< \infty$ and $v(R^{-})=\infty \Longleftrightarrow \sigma> \frac{p-\alpha}{p-1-\alpha}$.

\medskip

There are solutions $u(R^{-})=v(R^{-})=\infty \Longleftrightarrow 1<\sigma\leq \frac{p-\alpha}{p-1-\alpha}$.

\medskip

All solutions are bounded $\Longleftrightarrow \sigma \leq 1$. Since
$\delta \neq 0$, we rule out the possibility $\sigma=1$. Hence, all
positive radial solutions of \eqref{sys0} are bounded if and only if
$\sigma < 1$.

\medskip

Using the definition of $\sigma$ in \eqref{p9a} we conclude
(ii)-(iv). \hfill\qed

\medskip
\end{enumerate}

\section{Proof of Theorem \ref{thm2}}

Assume first that \eqref{nonct} holds. As argued in the beginning of the proof of Theorem \ref{thm1}
we are able to construct a non-constant positive radial solution in a maximal ball. By construction,
each component of such solution is increasing and by Theorem \ref{thm1}(ii) the solution is bounded.
Thus, the maximal domain of existence must be the whole space $\R^N$.

Conversely, assume that \eqref{nonct} does not hold and there exists
a non-constant global positive radial solution $(U,V)$ of
\eqref{sys0}. In order to reach a contradiction, we discuss
separately the following three cases.
\medskip

\noindent{\bf Case 1:} $0\leq \alpha< p-1$ and $mq> (p-1-\alpha)(p-1-\beta)$.

\medskip

From Theorem \ref{thm1} there exists a positive radial solution of \eqref{sys0} such that $v(1^-)=\infty$.

For any $\lambda>0$ set
$$
U_\lambda(x)=\lambda^{-1-\frac{p(m+1)-(1+\beta)}{\delta}}U(\lambda x)\,,\quad V_\lambda(x)=\lambda^{-\frac{p(p-1-\alpha)+q}{\delta}}V(\lambda x).
$$
Then, $(U_\lambda,V_\lambda)$ is a non-constant global positive radial global solution of \eqref{sys0}.
Replacing $(U,V)$ by $(U_\lambda,V_\lambda)$ for
$\lambda>0$ small enough, we may assume that $V(0)>v(0)$.

Let
\begin{equation}\label{rr}
R:=\sup\Big\{r\in (0,1): V(t)>v(t)\;\;\mbox{ in }(0,r)\Big\}.
\end{equation}
Clearly, since $V(0)>v(0)$, we have $0<R\leq 1$. We claim that $R=1$. Assuming the contrary, from \eqref{p3.1},
for all $0<r<R$ we obtain
$$
\Big[r^\gamma W^{p-1-\alpha}\Big]'=\frac{p-1-\alpha}{p-1} r^\gamma V^m(r)>\frac{p-1-\alpha}{p-1} r^\gamma v^m(r)=\Big[r^\gamma w^{p-1-\alpha}\Big]',
$$
where, as in the previous section we denote $W=U'$ and $w=u'$. An integration over $[0,r]$, $0<r\leq R$, yields
$W>w$ on $(0,R]$, which together with the second equation of \eqref{p3.1} implies
$$
\Big[r^{N-1} (V')^{p-1}\Big]'=r^{N-1} W^q(r)V^\beta(r)>r^{N-1} w^q(r)v^\beta(r)
=\Big[r^{N-1} (v')^{p-1}\Big]'\quad\mbox{ for all }0<r\leq R.
$$
As before, this leads to $V'>v'$ on $(0,R]$ and then $V>v$ on $[0,R]$ which contradicts the
maximality of $R$ in \eqref{rr}. Hence $R=1$, so $V>v$ on $(0,1)$. This yields $V(1^-)=\infty$ which is a
contradiction with the fact that $V$ is defined on the whole positive semiline.

\medskip

\noindent{\bf Case 2:} $\alpha> p-1$.

\medskip

By letting  $W=U'$ as in the proof of Lemma \ref{firstestim}, we have that $r\longmapsto r^{N-1}W^{p-1}$ is
nondecreasing, so there exists
$$
L:=\lim_{r\to \infty
} r^{N-1}W^{p-1}(r)\in (0,\infty].
$$
As in the proof of Theorem \ref{thm1}. we rewrite the first equation of \eqref{eqthm2} as
\begin{equation}\label{primgama}
\Big[r^{\gamma} W^{p-1-\alpha}\Big]'=\frac{\gamma}{N-1}r^{\gamma}V^{m}(r)\quad\mbox{ for all } r>0,
\end{equation}
where
$$
\gamma=\frac{(N-1)(p-1-\alpha)}{p-1}<0.
$$
Integrating in \eqref{primgama} over $[r,\infty]$, $r>0$, we find
$$
r^{\gamma}
W^{p-1-\alpha}(r)=L^{\frac{p-1-\alpha}{p-1}}+\frac{|\gamma|}{N-1}\int_r^\infty
t^\gamma V^m(t) dt\quad\mbox{ for all }r>0.
$$
Now, using the fact that $v$ is increasing we have
$$
r^{\gamma} W^{p-1-\alpha}(r) \geq  \frac{|\gamma|}{N-1} V^m(r)
\int_r^\infty t^\gamma dt\quad\mbox{ for all }r>0.
$$
In particular, the integral must be convergent, so $\gamma<-1$ and we deduce
\begin{equation}\label{wgama}
W^{p-1-\alpha}(r) \geq  \frac{\gamma}{(N-1)(\gamma+1)} rV^m(r)
\quad\mbox{ for all }r>0.
\end{equation}
We now use \eqref{wgama} into \eqref{primgama}. Since $\gamma<0$ we find
$$
\begin{aligned}
\frac{\gamma}{N-1}V^{m}(r)&=
\Big[W^{p-1-\alpha}\Big]'(r)+\frac{\gamma}{r} W^{p-1-\alpha}(r)\\
& \leq \Big[W^{p-1-\alpha}\Big]'(r)+\frac{\gamma^2}{(N-1)(\gamma+1)} V^m(r) \quad\mbox{ for all }r>0.
\end{aligned}
$$
Hence,
$$
\Big[W^{p-1-\alpha}\Big]'(r)\geq \frac{\gamma}{(N-1)(\gamma+1)} V^m(r)> 0 \quad\mbox{ for all }r>0.
$$
This shows that $W^{p-1-\alpha}$ is increasing, so $W$ must be decreasing. Since $W\geq 0$ and $W(0)=0$, it
follows that $W\equiv 0$, that is, $U\equiv U(0)>0$ is constant, contradiction.

\noindent{\bf Case 3:}  $\alpha=p-1$.

As above, $U'> 0$ and $V'> 0$ and from the first equation of \eqref{eqthm2} we find
$$
(p-1)\frac{U''(r)}{U'(r)}+\frac{N-1}{r}=V^m(r)\quad\mbox{ for all }r>0.
$$
Integrating over $[0,1]$ we deduce
$$
\ln\Big[r^{N-1}(U')^{p-1}(r)\Big]\Big|_{0}^{1}=\int_{0}^{1}V^m(t)dt,
$$
which is a contradiction since $U'(0)=0$ and the right-hand side of the above equality is finite.
\hfill\qed

\medskip

\noindent{\bf Remark.} The approach in Case 3 above shows in fact that if $\alpha=p-1$ then system \eqref{sys0}
has no non-constant positive radial solutions in any ball.
\medskip

\section{Proof of Theorem \ref{thm3}}

Assume $(u,v)$ is a non-constant global positive radial solution of
\eqref{sys0}. Let $t= \ln(r)\in \R$ and define the new functions
$X,Y,Z,W$ by
\begin{equation}\label{xyzw}
X(t)= \frac{ru'(r)}{u(r)},\;\;Y(t)= \frac{rv'(r)}{v(r)},\;\;Z(t)=\frac{rv^{m}(r)}{(u'(r))^{p-1-\alpha}} \mbox{ and }W(t)=\frac{rv^{\beta}(r)u'^{q}(r)}{v'^{p-1}(r)}.
\end{equation}
A direct calculation shows that $(X,Y,Z,W)$ satisfies
\begin{equation}\label{th2a}
\left\{
\begin{aligned}
&X_{t}= X\Big(\frac{p-N}{p-1}-X+\frac{1}{p-1}Z\Big) \quad\mbox{ for all } t\in \R, \\
&Y_{t}= Y\Big(\frac{p-N}{p-1}-Y+\frac{1}{p-1}W\Big) \quad\mbox{ for all } t\in \R, \\
&Z_{t}= Z\Big(\frac{(p-1)N-(N-1)\alpha}{p-1}-\frac{p-1-\alpha}{p-1}Z+mY\Big) \quad\mbox{ for all } t\in \R, \\
&W_{t}= W\Big(\frac{(p-1)N-q(N-1)}{p-1}+\beta Y+\frac{q}{p-1}Z-W\Big) \quad\mbox{ for all } t\in \R.
\end{aligned}
\right.
\end{equation}

By L'Hopital's rule we have
\begin{equation}\label{limit}
\lim_{t\rightarrow \infty}X(t)= \lim_{r\rightarrow \infty}\frac{ru'(r)}{u(r)}=\lim_{r\rightarrow \infty}\Big(1+\frac{ru''(r)}{u'(r)}\Big)=\lim_{t\rightarrow \infty}\Big(\frac{1}{p-1}Z(t)+\frac{p-N}{p-1}\Big),
\end{equation}
provided $\lim_{t\rightarrow \infty}Z(t)$ exists. In the following we shall study the system consisting of
the last three equations of \eq{th2a} which we write
\begin{equation}\label{th2b}
\zeta_{t}= g(\zeta) \quad\mbox{ in } \R,
\end{equation}
where
\begin{equation}\label{zet}
\zeta(t)=\left[\begin{array}{c}Y(t)\\Z(t)\\W(t)\end{array}\right] \quad\mbox{ and } \quad g(\zeta)=  \left[\begin{array}{c} Y\Big(\frac{p-N}{p-1}-Y+\frac{1}{p-1}W\Big)\\Z\Big(\frac{(p-1)N-(N-1)\alpha}{p-1}-\frac{p-1-\alpha}{p-1}Z+mY\Big) \\  W\Big(\frac{(p-1)N-q(N-1)}{p-1}+\beta Y+\frac{q}{p-1}Z-W\Big) \end{array}\right].
\end{equation}
Note that the system \eq{th2b} is cooperative and irreducible as described in the Appendix.
Also, the only equilibrium point of \eq{th2b}-\eqref{zet} with all components being strictly positive is
\begin{equation}\label{eqinf}
P_\infty=\left[\begin{array}{c}Y_\infty\\Z_\infty\\W_\infty\end{array}\right],
\end{equation}
where
\begin{equation}\label{th2c}
\left\{
\begin{aligned}
&\frac{p-N}{p-1}-Y_\infty+\frac{1}{p-1}W_\infty=0,  \\
&\frac{(p-1)N-(N-1)\alpha}{p-1}-\frac{p-1-\alpha}{p-1}Z_\infty+mY_\infty=0, \\
&\frac{(p-1)N-q(N-1)}{p-1}+\beta Y_\infty+\frac{q}{p-1}Z_\infty-W_\infty=0.
\end{aligned}
\right.
\end{equation}
Solving \eqref{th2c} we find
\begin{equation}\label{th2d}
\left\{
\begin{aligned}
&Y_\infty=\frac{p(p-1-\alpha)+q}{\delta}, \\
&Z_\infty=\frac{m(p-1)}{p-1-\alpha}Y_\infty+N+\frac{\alpha}{p-1-\alpha}, \\
&W_\infty= (p-1)Y_\infty+N-p.
\end{aligned}
\right.
\end{equation}

\begin{lemma}\label{l2b}
The equilibrium point $P_\infty$ is asymptotically stable.
\end{lemma}
\begin{proof}
Using \eq{th2c} we compute the linearized matrix of \eq{th2b} at $P_\infty$  as
$$
M_\infty=\left[\begin{array}{ccc}-Y_\infty&0&\frac{1}{p-1}Y_\infty\\mZ_\infty&-\frac{p-1-\alpha}{p-1}Z_\infty&0\\\beta W_\infty&\frac{q}{p-1}W_\infty&-W_\infty\end{array}\right].
$$
The characteristic polynomial of $M_\infty$ is
$$
P(\lambda)=\det(\lambda I-M)=\lambda^{3}+a\lambda^{2}+b \lambda+c,
$$
where
\begin{equation*}
\left\{
\begin{aligned}
a&=Y_\infty+\frac{p-1-\alpha}{p-1}Z_\infty+W_\infty, \\
b&=\frac{p-1-\alpha}{p-1}Y_\infty Z_\infty+\frac{p-1-\beta}{p-1}Y_\infty W_\infty+\frac{p-1-\alpha}{p-1}Z_\infty W_\infty, \\
c&= \frac{\delta}{(p-1)^2}Y_\infty Z_\infty W_\infty.
\end{aligned}
\right.
\end{equation*}
Since $Y_\infty$, $Z_\infty$, $W_\infty>0$ and $p-1-\beta>0$ (which follows easily from $\delta>0$) we have
$$
a\geq \frac{p-1-\beta}{p-1}Y_\infty+\frac{p-1-\alpha}{p-1}Z_\infty+\frac{p-1-\beta}{p-1}W_\infty.
$$
Thus, by AM-GM inequality we find
$$
a\geq 3\frac{(p-1-\alpha)^{\frac{1}{3}}(p-1-\beta)^{\frac{2}{3}}}{p-1} (Y_\infty Z_\infty W_\infty)^{\frac{1}{3}}.
$$
Similarly, by AM-GM we obtain
$$
b\geq
3\frac{(p-1-\alpha)^{\frac{2}{3}}(p-1-\beta)^{\frac{1}{3}}}{p-1} (Y_\infty Z_\infty W_\infty)^{\frac{2}{3}}.
$$
We now multiply the above estimates to deduce
$$
ab\geq 9 \frac{(p-1-\alpha)(p-1-\beta)}{(p-1)^2} (Y_\infty Z_\infty W_\infty)>9c.
$$

\medskip

We claim that all three roots $\lambda_1$, $\lambda_2$ and $\lambda_3$ of the characteristic
polynomial $P(\lambda)$ of $M_\infty$ have negative real part. Indeed, if $\lambda_i \in \R$, for all
$i=1$, $2$, $3$ then, since $P(\lambda)>0$ for all $\lambda \geq 0$ it follows that $\lambda_i< 0$ for all
$i=1$, $2$, $3$. If $P$ has exactly one real root, say $\lambda_1 \in \R$, then ${\rm Re}(\lambda_2)={\rm Re}(\lambda_3)$. Using $P(-a)=-ab+c<0$, it follows that $\lambda_1> -a$. Since $\lambda_1+\lambda_2+\lambda_3= -a$ we easily deduce that ${\rm Re}(\lambda_2)={\rm Re}(\lambda_3)<0$. This proves that $P_\infty$ is asymptotically stable.

\end{proof}

\medskip

The following result is crucial in our analysis to establish the behavior of $\zeta(t)$ as $t\rightarrow \infty$.

\begin{lemma}\label{l2c}
For all $t\in \R$ we have
\begin{equation}\label{infy}
0<Y(t)< Y_\infty\,,\qquad N+\frac{\alpha}{p-1-\alpha}<Z(t)<Z_\infty\,,\qquad N<W(t)<W_\infty.
\end{equation}
\end{lemma}
\begin{proof}
We divide our arguments into four steps.
\medskip

\noindent{\it Step 1: \it Preliminary Facts:} $Z(t)> N+\frac{\alpha}{p-1-\alpha}$, $W(t)> N$ for all $t\in \R$
and $\lim_{t\rightarrow -\infty}Y(t)=0$.

\medskip

The lower bounds for $Z$ and $W$ follow from \eqref{l01} and \eqref{l02} in Lemma \ref{firstestim}.
Since $v'(0)=0$ and $v(0)>0$ we have $\lim_{t\rightarrow -\infty}Y(t)= \lim_{r\rightarrow 0}\frac{rv'(r)}{v(r)}= 0$.

\bigskip

\noindent{\it Step 2: There exists $T\in \R$ such that $Z(t)< Z_\infty$ for all $t\in (-\infty, T]$.}
\medskip

The conclusion of this Step follows immediately once we prove that
\begin{equation}\label{zlim}
\lim_{t\to-\infty} Z(t)=N+\frac{\alpha}{p-1-\alpha}.
\end{equation}

Let $t\in (-\infty, 0)$ and $r= e^t \in (0, 1)$. We use the Generalized Mean Value
Theorem\footnote{Generalized Mean Value Theorem (or Cauchy's Theorem) states that if
$f$,$g: [a, b]\rightarrow \R$ are differentiable functions on $(a, b)$ and continuous on $[a, b]$,
then there exists $c\in (a, b)$ such that $\frac{f(b)-f(a)}{g(b)-g(a)}= \frac{f'(c)}{g'(c)}$.}
\cite[Theorem 5.9, page 107]{R1976}  over the interval $[0, r]$. Thus, there exists $c\in (0, r)$ such that
\begin{equation*}
\begin{aligned}
Z(t)&= \frac{rv^{m}(r)}{u'^{p-1-\alpha}(r)}=\frac{\displaystyle\frac{d}{dr}\Big[rv^{m}(r)\Big](c)}{\displaystyle \frac{d}{dr}\Big[u'^{p-1-\alpha}(r)\Big](c)}\\
&= \frac{v^m(c)+cmv^{m-1}(c)v'(c)}{(p-1-\alpha)u'^{(p-2-\alpha)}(c)u''(c)}\\
&= \frac{p-1}{p-1-\alpha}\frac{v^m(c)u'^{\alpha}(c)+cmv^{m-1}(c)v'(c)u'^{\alpha}(c)}{(u'^{p-1})'(c)}.
\end{aligned}
\end{equation*}
Using the first equation in \eq{p1} we find
$$
Z(t)= \frac{p-1}{p-1-\alpha}\frac{v^m(c)u'^{\alpha}(c) \Big(1+m\frac{cv'(c)}{v(c)}\Big)}{\displaystyle v^m(c)u'^{\alpha}(c)-\frac{N-1}{c}u'^{p-1}(c)},
$$
and so,
\begin{equation}\label{th2f}
Z(t)= \frac{p-1}{p-1-\alpha}\frac{Z({\rm ln}c)}{Z({\rm ln}c)-(N-1)} \Big(1+mY({\rm ln}c)\Big).
\end{equation}
Since
\begin{equation}\label{z1}
Z(t)>N+\frac{\alpha}{p-1-\alpha}\quad\mbox{ for all }\quad t\in\R,
\end{equation}
 we have
$$
\frac{Z({\rm ln}c)}{Z({\rm ln}c)-(N-1)}<\Big( N+\frac{\alpha}{p-1-\alpha}\Big)\frac{p-1-\alpha}{p-1}.
$$
Thus from \eq{th2f} we find
\begin{equation}\label{z2}
\lim{\rm sup}_{t\rightarrow -\infty}Z(t)\leq \Big( N+\frac{\alpha}{p-1-\alpha}\Big) \big(1+\lim_{t\rightarrow -\infty}Y(t)\big)= N+\frac{\alpha}{p-1-\alpha}.
\end{equation}
Now, combining \eqref{z1} and \eqref{z2} we obtain that \eqref{zlim} holds.
It follows that there exists $T\in \R$ such that $Z(t)< Z_\infty$ for all $t\leq T$.

\bigskip

\noindent{\it Step 3:
There exists a sequence $t_j \rightarrow -\infty$ such that
\begin{equation}\label{tj}
Y(t_j)< Y_\infty\,, \quad Z(t_j)< Z_\infty \quad \mbox{ and } \quad W(t_j)< W_\infty \quad\mbox{ for all } j\geq 1.
\end{equation}
}
Assume the above assertion is not true. In view of the previous steps and by taking $T\in \R$ found at Step 2
small enough, we may assume
\begin{equation}\label{asump}
Y(t)< Y_\infty\,, \quad Z(t)< Z_\infty \quad \mbox{ and }\quad W(t)\geq W_\infty \quad\mbox{ for all } t\in (-\infty, T].
\end{equation}
Using this fact and last equation in \eq{th2a} we deduce $W_t< 0$ on $(-\infty, T]$. Hence, $W$ is decreasing
in a neighbourhood of $-\infty$ and there exists
$$
L:= \lim_{t\rightarrow -\infty} W(t)= \lim_{r\rightarrow 0}\frac{rv^{\beta}(r)u'^{q}(r)}{v'^{p-1}(r)}.
$$
Let $t\in (-\infty, T]$ and $r= e^{t}$. Applying the Generalized Mean Value Theorem as in the previous step
and using the second equation of \eqref{p1} we find $c\in (0, r)$ such that
\begin{equation*}
\begin{aligned}
W(t)&= \frac{rv^{\beta}(r)u'^{q}(r)}{v'^{p-1}(r)}= \frac{\displaystyle \frac{d}{dr}\Big[rv^{\beta}(r)u'^{q}(r)\Big](c)}{\displaystyle \frac{d}{dr}\Big[v'^{p-1}(r)\Big](c)}\\
&=\frac{v^{\beta}(c)u'^{q}(c)+\beta cv^{\beta-1}(c)v'(c)u'^{q}(c)+qcv^{\beta}(c)u'^{q-1}(c)u''(c)}{v^{\beta}(c)u'^{q}(c)-\frac{N-1}{c}v'^{p-1}(c)}.
\end{aligned}
\end{equation*}
Using the first equation of \eq{p1} we further compute
\begin{equation}\label{th2g}
W(t)= \frac{W({\rm ln}c)}{W({\rm ln}c)-(N-1)}\left[1+\beta Y({\rm ln}c)+\frac{q}{p-1}Z({\rm ln}c)-\frac{q(N-1)}{p-1}\right].
\end{equation}
Recall that by Step $1$, we have $Z> N$ and $W> N$ so that right hand side of \eq{th2g} is positive.
Passing to the limit with $t\rightarrow -\infty$ (note that this implies $c\rightarrow 0$) and using
$\lim_{t\rightarrow -\infty}Z(t)< Z_\infty$ and $\lim_{t\rightarrow -\infty}Y(t)= 0$  we find from \eq{th2g} that
\begin{equation}\label{th2h}
L= \lim_{t\rightarrow -\infty}W(t)\leq \frac{L}{L-(N-1)}\left[1+\frac{q}{p-1}Z_\infty-\frac{q(N-1)}{p-1}\right].
\end{equation}
Hence, by \eq{th2h} and the last equation of \eq{th2c} we find
$$
\begin{aligned}
L&\leq N+\frac{q}{p-1}Z_\infty-\frac{q(N-1)}{p-1}\\
&< \frac{N(p-1)-q(N-1)}{p-1}+\beta Y_\infty+\frac{q}{p-1}Z_\infty= W_\infty.
\end{aligned}
$$
Thus $L< W_\infty$ which, in light of the fact that $W$ is decreasing on $(-\infty, T]$ implies $W(t)< W_\infty$
for all $t\in (-\infty, T]$, a contradiction with \eqref{asump}.
\bigskip

\noindent{\it Step 4:  Conclusion of the proof. }
\medskip

Using the comparison result in Theorem \ref{comparis} on each of the intervals $[t_j,\infty)$ we deduce
\begin{equation}\label{et}
Y(t)< Y_\infty\,, \quad Z(t)< Z_\infty \quad \mbox{ and } \quad W(t)< W_\infty \quad\mbox{ for all } t\geq t_j.
\end{equation}
Since $t_j\to -\infty$ it follows that the estimates in \eqref{et} hold for all $t\in \R$ and this together with
Step 1 proves \eqref{infy}.
\end{proof}

\noindent{\bf Proof of Theorem \ref{thm3}
completed}. Let $(u,v)$ be a non-constant global positive radial
solution of \eqref{sys0}. Denote by $(X,Y,Z,W)$ the solution of
\eqref{th2a} corresponding to $u$ and $v$ as described in
\eqref{xyzw}. Then $
\zeta(t)=\left[\begin{array}{c}Y(t)\\Z(t)\\W(t)\end{array}\right] $
is a solution of \eqref{th2b}-\eqref{zet}. Thus, by  Lemma \ref{l2c}
we have
$$
P_*:=\left[\begin{array}{c}0\\ \displaystyle N+\frac{\alpha}{p-1-\alpha}\\N\end{array}\right]<\zeta(0).
$$
By Theorem \ref{lebesgue} there exists a set $\Sigma\subset \R^3$ of Lebesgue measure zero  such that
\begin{equation}\label{omegaa}
\omega(\widetilde P)\subseteq E\quad\mbox{ for all }\quad \widetilde P\in [P_*,P_\infty]\setminus\Sigma,
\end{equation}
where $E$ is the set of equilibrium points associated with \eqref{th2b}-\eqref{zet}.
For $\widetilde P\in [P_*,P_\infty]\setminus\Sigma$ denote by
$$
\Phi(t,\widetilde P)=\left[\begin{array}{c}\widetilde{Y}(t)\\\widetilde{Z}(t)\\\widetilde{W}(t)\end{array}\right]
$$
the flow of \eqref{th2b}  associated with the initial data $\widetilde P$. Since $\widetilde P\geq P_*$, by the
comparison result in Theorem \ref{comparis} it follows that
$$
\Phi(t,\widetilde P)\geq \left[\begin{array}{c}0\\ \displaystyle N+\frac{\alpha}{p-1-\alpha}\\0\end{array}\right]\quad\mbox{ for all }t\geq 0.
$$
Therefore, the only equilibrium points that $\omega(\widetilde P)$ may approach must be non-negative and
have the second component greater than or equal to $N+\frac{\alpha}{p-1-\alpha}$. It follows that
$$
\omega(\widetilde P)\subseteq \{P_1,P_2,P_3,P_\infty\},
$$
where
$$
P_1=\left[\begin{array}{c}0\\ \displaystyle
N+\frac{\alpha}{p-1-\alpha}\\0\end{array}\right]\;,\quad
P_2=\left[\begin{array}{c}0\\ \displaystyle
N+\frac{\alpha}{p-1-\alpha}\\
\displaystyle N+\frac{q}{p-1-\alpha}\end{array}\right]\;,\quad
P_3=\left[\begin{array}{c}\displaystyle \frac{p-N}{p-1}\\
\displaystyle
N+\frac{\alpha+m(p-N)}{p-1-\alpha}\\0\end{array}\right]
$$
and $P_\infty$ is given by \eqref{eqinf}. Note, that $P_3$ has all components non-negative if and only of $p\geq N$.

We claim that
\begin{equation}\label{eqfin}
\omega(\widetilde P)=\{P_\infty\}\quad\mbox{ for all }\quad
\widetilde P\in [P_*,P_\infty]\setminus\Sigma.
\end{equation}

First we note that if $P_\infty\in \omega(\widetilde P)$ then, since $P_\infty$ is asymptotically stable, it follows
that  $\omega(\widetilde P)=\{P_\infty\}$. Assume in the following that $P_\infty\not \in \omega(\widetilde P)$ so
$\omega(\widetilde P)\subseteq \{P_1,P_2,P_3\}$.

If $\{P_1,P_2\}\subset \omega(\widetilde P)$ or $\{P_2,P_3\}\subset \omega(\widetilde P)$ then $\widetilde W$
converges along a subsequence to 0 and to
$N+\frac{q}{p-1-\alpha}$.  By the Intermediate Value Theorem we deduce that for all $0<\tau<N+\frac{q}{p-1-\alpha}$
there exists a sequence $t_j\to \infty$ such that $\widetilde W(t_j)=\tau$ which contradicts the fact that
$\omega(\tilde P)$ is finite. Similarly, if $\{P_1,P_3\}\subset \omega(\widetilde P)$ we deduce that $p>N$
and for all $0<\gamma<\frac{p-N}{p-1}$ there exists a sequence $t_j\to \infty$ such that $\widetilde Y(t_j)=\gamma$
which is again a contradiction.

It follows that $\omega(\widetilde P)$ is a singleton. Let us show that in this situation we again raise a
contradiction. Indeed, if for instance $\omega(\widetilde P)=\{P_2\}$ then, as $t\to \infty$, we have
$$
\widetilde Y(t)\to 0\,, \quad \widetilde Z(t)\to N+\frac{\alpha}{p-1-\alpha}\quad\mbox{and}\quad \widetilde W(t)\to N+\frac{q}{p-1-\alpha}.
$$
Then, for large $t>0$ one has
$$
\widetilde Y_t=\widetilde Y\Big(\frac{p-N}{p-1}-\widetilde Y+\frac{1}{p-1}\widetilde W\Big)>0
$$
so $\widetilde Y$ is increasing in a neighbourhood of infinity. It
follows that for large $t>0$ we have $\widetilde Y(t)\leq \lim_{s\to
\infty}\widetilde Y(s)=0$, contradiction. Similarly, if
$\omega(\widetilde P)=\{P_1\}$ or if $\omega(\widetilde P)=\{P_3\}$
we reach a contradiction. Hence, the claim \eqref{eqfin} holds.

Take now $P\in [[P_*,P_\infty]]\cap \Sigma$ and let
$\widetilde P\in [P_*,P_\infty]\setminus \Sigma$ be such that
$\widetilde P<P$. By \eqref{eqfin} and the Dichotomy Theorem
\ref{dich} we have
\begin{itemize}
\item either $\{P_\infty\}=\omega(\widetilde P)<\omega(P)$;
\item or $\omega(P)=\omega(\widetilde P)=\{P_\infty\}$.
\end{itemize}
The first alternative cannot hold since by the comparison result in
Theorem \ref{comparis} we have $\omega(P)\leq P_\infty$. It follows
that $\omega(P)=\{ P_\infty \}$ so,
$$
\omega(P)=\{P_\infty\}\quad\mbox{ for all } P\in [[P_*, P_\infty]].
$$
In particular, for $P=\zeta(0)$ we find
$$
\omega(\zeta(0))=\{P_\infty\}.
$$
Thus, as $t\to \infty$ we have
$$
Y(t)\to Y_\infty\,,\qquad Z(t)\to Z_\infty\,,\qquad W(t)\to W_\infty.
$$

By \eq{limit} there exists
$$
X_\infty:= \lim_{t\rightarrow \infty}X(t)= \frac{1}{p-1}Z_\infty+\frac{p-N}{p-1}.
$$
Observe that
$$
\frac{u^{\delta}(r)}{r^{p(m+1)-(1+\beta)+\delta}}= \frac{1}{Y^{m(p-1)}(t)Z^{p-1-\beta}(t)W^{m}(t)X^{\delta}(t)} \quad\mbox{ for all } r> 0,
$$
so
\begin{equation}\label{th2i}
\lim_{r\rightarrow \infty} \frac{u(r)}{r^{1+\frac{p(m+1)-(1+\beta)}{\delta}}}= A,
\end{equation}
where
$$
A= \frac{1}{Y_\infty^\frac{m(p-1)}{\delta}Z_\infty^{\frac{p-1-\beta}{\delta}}W_\infty^{\frac{m}{\delta}}X_\infty}\in (0, \infty).
$$
Similarly, we find
\begin{equation}\label{th2j}
\lim_{r\rightarrow \infty} \frac{v(r)}{r^{\frac{p(p-1-\alpha)+q}{\delta}}}= B,
\end{equation}
where
$$
B= \frac{1}{Y_\infty^\frac{(p-1)(p-1-\alpha)}{\delta}Z_\infty^{\frac{q}{\delta}}W_\infty^{\frac{p-1-\alpha}{\delta}}}\in (0, \infty).
$$

\section{Proof of Theorem \ref{thm4}}
The existence of a non-constant global positive radial solution
$(u,v)$ of \eqref{sys0} with $u(0)=a>0$ and $v(0)=b>0$ follows from
Theorem \ref{thm1}. First, there exists a non-constant local
positive radial solution $(u,v)$ as above as a fixed point of the
mapping given by \eqref{t1}-\eqref{t3}. In light of Theorem
\ref{thm1} (ii), such a solution must be global. We focus in the
following on the uniqueness part.

For any non-constant positive global solution $(u, v)$ of system
\eq{sys0} we denote
\begin{equation}\label{th3a}
u(r)= U(t), \;\;\; v(r)= V(t) \mbox{ where } r=t^{\theta}, \;\; \theta= -\frac{p-1}{N-p}< 0.
\end{equation}
Then $(U, V)$ satisfies (throughout this section $'$ denotes the
derivative with respect to $t$ variable)

\begin{equation}\label{th3b} \left\{
\begin{aligned}
&\Big[|U'(t)|^{p-\alpha-2}U'(t)\Big]'=\frac{p-1-\alpha}{p-1}|\theta|^{p-\alpha}t^{(\theta-1)(p-\alpha)}V^{m}(t) &&\quad\mbox{ for all } t>0,\\
&\Big[|V'(t)|^{p-2}V'(t)\Big]'=|\theta|^{p-q}t^{(\theta-1)(p-q)}V^{\beta}(t)|U'(t)|^{q} &&\quad\mbox{ for all } t>0,\\
& U'(t)<0, \; V'(t)<0,\; U(t)>0,\; V(t)>0  &&\quad\mbox{ for all } t>0,\\
&U'(\infty)=V'(\infty)=0, \; U(\infty)=u(0), \; V(\infty)=v(0).
\end{aligned}
\right.
\end{equation}
Letting $W(t)= |U'(t)|^{p-\alpha-2}U'(t)$ we transform \eq{th3b} into
\begin{equation}\label{th3c}
\left\{
\begin{aligned}
&W'(t)=\frac{p-1-\alpha}{p-1}|\theta|^{p-\alpha}t^{(\theta-1)(p-\alpha)}V^{m}(t) &&\quad\mbox{ for all } t>0,\\
&\Big[|V'(t)|^{p-2}V'(t)\Big]'=|\theta|^{p-q}t^{(\theta-1)(p-q)}V^{\beta}(t)|W(t)|^{\frac{q}{p-1-\alpha}} &&\quad\mbox{ for all } t>0,\\
& W(t)<0, \; V'(t)<0,\; V(t)>0  &&\quad\mbox{ for all } t>0,\\
&W(\infty)= 0, V'(\infty)=0, V(\infty)=v(0).
\end{aligned}
\right.
\end{equation}
Let now $a, b>0$ and $(u, v)$, $(\widetilde{u}, \widetilde{v})$ be
two pairs of non-constant global positive radial solutions of
\eq{sys0} with $u(0)= \widetilde{u}(0)= a$ and $v(0)=
\widetilde{v}(0)= b$. We want to show that $u\equiv \widetilde{u}$
and $v\equiv \widetilde{v}$.

\medskip

Let $\epsilon> 0$ and set
$$
\widehat{u}(r)= (1+\epsilon)u(r), \;\;\;\; \widehat{v}(r)= (1+\epsilon)^{\frac{p-1-\alpha}{m}}v(r).
$$
It follows that (with $r=t^\theta$ from \eqref{th3a})
\begin{equation}\label{s1}
\widehat{U}(t)= \widehat{u}(r), \; \widehat{V}(t)= \widehat{v}(r) \mbox{ and } \widehat{W}(t)= | \widehat{U}'(t)|^{p-\alpha-2} \widehat{U}'(t)
\end{equation}
satisfy
\begin{equation}\label{th3d}
\left\{
\begin{aligned}
&\widehat{W}'(t)=\frac{p-1-\alpha}{p-1}|\theta|^{p}t^{(\theta-1)(p-\alpha)}\widehat{V}^{m}(t) &&\quad\mbox{ for all } t>0,\\
&\Big[|\widehat{V}'(t)|^{p-2}\widehat{V}'(t)\Big]'= (1+\epsilon)^{\frac{\delta}{m}}|\theta|^{p-q}t^{(\theta-1)(p-q)}\widehat{V}^{\beta}(t)|\widehat{W}(t)|^{\frac{q}{p-1-\alpha}} &&\quad\mbox{ for all } t>0,\\
& \widehat W(t)<0, \; \widehat V'(t)<0,\; \widehat V(t)>0  &&\quad\mbox{ for all } t>0,\\
&\widehat{W}(\infty)= 0, \widehat{V}'(\infty)=0, \widehat{V}(\infty)= (1+\epsilon)^{\frac{p-1-\alpha}{m}}v(0).
\end{aligned}
\right.
\end{equation}
Through the same change of variable $r= t^{\theta}$ given by \eqref{th3a}, the functions
\begin{equation}\label{s2}
\widetilde{U}(t)= \widetilde{u}(r), \; \widetilde{V}(t)= \widetilde{v}(r) \mbox{ and } \widetilde{W}(t)= | \widetilde{U}'(t)|^{p-\alpha-2} \widetilde{U}'(t),
\end{equation}
satisfy
\begin{equation}\label{th3e}
\left\{
\begin{aligned}
&\widetilde{W}'(t)=\frac{p-1-\alpha}{p-1}|\theta|^{p-\alpha}t^{(\theta-1)(p-\alpha)}\widetilde{V}^{m}(t) &&\quad\mbox{ for all } t>0,\\
&\Big[|\widetilde{V}'(t)|^{p-2}\widetilde{V}'(t)\Big]'= |\theta|^{p-q}t^{(\theta-1)(p-q)}\widetilde{V}^{\beta}(t)|\widetilde{W}(t)|^{\frac{q}{p-1-\alpha}} &&\quad\mbox{ for all } t>0,\\
& \widetilde W(t)<0, \; \widetilde V'(t)<0,\; \widetilde V(t)>0  &&\quad\mbox{ for all } t>0,\\
&\widetilde{W}(\infty)= 0, \widetilde{V}'(\infty)=0, \widetilde{V}(\infty)= \widetilde{v}(0)= b.
\end{aligned}
\right.
\end{equation}
Since, $\widehat{V}(\infty)> \widetilde{V}(\infty)$ it follows from
the first equation of \eqref{th3d} and \eqref{th3e} that the set
$$
A:= \{t>0: \;\; \widehat{W}'>\widetilde{W}'  \mbox{ on } (t, \infty)\}
$$
is nonempty. We claim that $A= (0, \infty)$. Assuming the contrary, one has
$$
t_0= \inf A> 0
$$ together with
\begin{equation}\label{th3f}
\widehat{W}'(t)>\widetilde{W}'(t) \quad\mbox{ for all } t\in (t_0, \infty)\quad\mbox{ and }\quad
\widehat{W}'(t_0)= \widetilde{W}'(t_0).
\end{equation}
Using the first equation in \eq{th3d} and \eq{th3e} it follows that
\begin{equation}\label{th3g}
\widehat{V}(t)>\widetilde{V}(t) \quad\mbox{ for all } t\in (t_0, \infty)
\quad\mbox{ and }\quad
\widehat{V}(t_0)= \widetilde{V}(t_0).
\end{equation}
Integrating \eq{th3f} we find  (since $\widehat{W}(\infty)= \widetilde{W}(\infty)= 0$) that
\begin{equation}\label{th3h}
|\widehat{W}(t)|= -\widehat{W}(t)> -\widetilde{W}(t)= |\widetilde{W}(t)| \quad\mbox{ for all } t\in (t_0, \infty).
\end{equation}
Hence, from second equation of \eq{th3d}, \eq{th3e} and from \eq{th3g}, \eqref{th3h} we deduce
\begin{equation*}
\Big[|\widehat{V}'(t)|^{p-2}\widehat{V}'(t)\Big]'> \Big[|\widetilde{V}'(t)|^{p-2}\widetilde{V}'(t)\Big]' \quad\mbox{ for all } t\in (t_0, \infty).
\end{equation*}
An integration over $[t, \infty]$ in the above inequality yields
\begin{equation*}
|\widehat{V}'(t)|^{p-1}= -|\widehat{V}'(t)|^{p-2}\widehat{V}'(t)> -|\widetilde{V}'(t)|^{p-2}\widetilde{V}'(t)= |\widetilde{V}'(t)|^{p-1} \quad\mbox{ for all } t\in (t_0, \infty).
\end{equation*}
This implies
$$
-\widehat{V}'(t)> -\widetilde{V}'(t) \quad\mbox{ for all } t\in (t_0, \infty).
$$
Integrating now over  $[t_0, \infty]$ and using $\widehat{V}(\infty)> \widetilde{V}(\infty)$ we obtain
$$
\widehat{V}(t_0)> \widetilde{V}(t_0)+ (\widehat{V}(\infty)- \widetilde{V}(\infty))> \widetilde{V}(t_0),
$$
which contradicts \eq{th3g}. Hence $A= (0, \infty)$ which shows that $\widehat{W}'(t)> \widetilde{W}'(t)$ for all $t\in (0, \infty)$. Integrating over $[t, \infty]$ we have
$$
|\widehat{W}(t)|= -\widehat{W}(t)> -\widetilde{W}(t)= |\widetilde{W}(t)| \quad\mbox{ for all }  t\in (0, \infty).
$$
Using this estimate and the expression of $\widehat{W}$ and $\widetilde{W}$ in \eq{s1} and \eq{s2} respectively we find
$$
-\widehat{U}'(t)=|\widehat{U}'(t)|> |\widetilde{U}'(t)|=-\widetilde{U}'(t) \quad\mbox{ for all }  t\in (0, \infty).
$$
A further integration over $[t,\infty]$ yields
$$
\widehat{U}(t)> \widetilde{U}(t)+\widehat{U}(\infty)-\widetilde{U}(\infty)> \widetilde{U}(t) \quad\mbox{ for all }  t\in (0, \infty).
$$
This implies
$$
\widehat{u}(r)= (1+\epsilon)u(r)> \widetilde{u}(r) \quad\mbox{ for all } r> 0.
$$
Passing to the limit with $\epsilon \rightarrow 0$ we find $u\geq \widetilde{u}$ in $(0, \infty)$. Also, $\widehat{W}'> \widetilde{W}'$ in $(0, \infty)$ together with \eq{th3d} and \eq{th3e} yield $\widehat{V}> \widetilde{V}$ in $(0, \infty)$. So,
$$
(1+\epsilon)^{\frac{p-1-\alpha}{m}}v(r)> \widetilde{v}(r) \quad\mbox{ for all } r>0.
$$
This also entails (by letting $\epsilon \rightarrow 0$) that $v\geq \widetilde{v}$ in $(0, \infty)$. Now,
we can replace $u$ by $\widetilde{u}$, $v$ by $\widetilde{v}$  to deduce
$$
\widetilde{u}\geq u, \;\;\; \widetilde{v}\geq v \quad \mbox{ in }\quad  (0, \infty).
$$
Thus, $u\equiv \widetilde{u}$ and $v\equiv \widetilde{v}$. This concludes the proof.

%\appendix

\begin{appendices}
\section{Some results for cooperative dynamical systems}

We recall here some results on dynamical systems that we used in the current work.

For any vectors $x=(x_1,x_2,x_3), y=(y_1,y_2,y_3)\in \R^3$ we let
$$
x\leq y\quad\mbox{ if }\quad x_i\leq y_i\,,\;\; i=1,2,3,
$$
$$
x< y\quad\mbox{ if }\quad x_i< y_i\,,\;\; i=1,2,3.
$$
We also define the
the closed interval  $[x,y]=\{u\in \R^3:x\leq u\leq y\}$ and the open interval
$ [[x,y]]=\{u\in \R^3:x\leq u\leq y\}$ with endpoints at $x$ and $y$.

A set $X\subset \R^3$ is said to be $p$-convex if for any $x,y\in X$, the segment line
joining $x$ and $y$ is a subset of $X$. Throughout this section $X$ will be an open $p$-convex subset of $\R^3$.

Let $g:X\to \R^3$ be a $C^1$-vector field. For any $P\in \R^3$ we denote by $\Phi(t,P)$ the maximally defined
solution of the differential equation
\begin{equation}\label{flow}
\frac{d\zeta}{dt}=g(\zeta)
\end{equation}
subject to the initial condition $\zeta(0)=P$. The collection of maps$\{\Phi(t,\cdot)\}$ is called the flow of
the differential equation \eqref{flow}.

\begin{de}
A $C^1$-vector field $g:X\to \R^3$ is said to be cooperative if at any point $P\in X$ we have
$$
\frac{\partial g_i}{\partial x_j}(P)\geq 0\quad \mbox{  for any }\; i,j=1,2,3, \;\; i\neq j.
$$
\end{de}
Cooperative systems enjoy a comparison property of the flows as stated below.

\begin{theorem}\label{comparis}{\rm (See \cite{H1985})}
Assume the $C^1$-vector field $g:X\to \R^3$ is cooperative and let $\zeta, \xi:[0,a]\to \R$, $a>0$, be two
solutions of \eqref{flow} such that
$$
\zeta(0)<\xi(0)\quad (\mbox{ resp.} \zeta(0)\leq \xi(0)\;).
$$
Then
$$
\zeta(t)<\xi(t)\quad (\mbox { resp.}  \zeta(t)\leq \xi(t)\;) \quad \mbox{ for all } t\in [0,a].
$$
\end{theorem}

\begin{de}
The equilibrium set of \eqref{flow} is the set $E$ of points $P\in
X$ such that $g(P)=0$. Any such element is called an equilibrium
point of \eqref{flow}. Obviously, $\Phi(t,P)=P$ for any equilibrium
point $P$.
\end{de}

\begin{de}
Let $P\in X$. The $\omega$-limit set $\omega(P)$ is defined as the set of all points $Q\in \R^3$ such that
there exists $\{t_j\}$, $t_j\to \infty$ (as $j\to \infty$) such that $\Phi(t_j, P)\to Q$ (as $j\to \infty$).
\end{de}

The following dichotomy result obtained in \cite{H1985} essentially states that the omega limit sets preserve the
partial order between the elements of $X$ or approach the equilibrium set $E$.

\begin{theorem}\label{dich}{\rm (Limit Set Dichotomy, see \cite[Theorem 3.8]{H1985}, \cite[Theorem 1.16]{HS2005})}

Assume the $C^1$-vector field $g:X\to \R^3$ is cooperative and let $P,Q\in X$, $P<Q$.
Then the following alternative holds:
\begin{enumerate}
\item[(i)] either $\omega(P)<\omega(Q)$;
\item[(ii)] or $\omega(P)=\omega(Q)\subset E$.
\end{enumerate}

\end{theorem}

\begin{de}
A $C^1$-vector field $g:X\to \R^3$ is said to be irreducible if at any point $P\in X$ its gradient $\nabla g(P)$ is
an irreducible matrix.
\end{de}

\begin{remark}
Recall that a general $n\times n$ matrix $M$ is irreducible if one of the following equivalent conditions holds:
\begin{enumerate}
\item[(i)] for any nontrivial partition $I\cup J$ of the set $\{1,2\dots,n\}$ there exists $i\in I$, $j\in J$ such that $M_{ij}\neq 0$;
\item[(ii)] the digraph associated with $M$, that is, the oriented graph with vertices at $1,2,\dots,n$ which connects $(i,j)$ if and only if $M_{ij}\neq 0$, is strongly connected.
\end{enumerate}
\end{remark}

The compact omega limit sets of cooperative and irreducible vector fields have a particular property in the
sense that they approach the equilibrium set for almost all points in $X$. This is formulated in the result below.
\begin{theorem}\label{lebesgue}{\rm (See \cite[Theorem 4.1]{H1985})}

Assume the $C^1$-vector field $g:X\to \R^3$ is cooperative and
irreducible and that for all $P\in X$ the $\omega$-limit set
$\omega(P)$ is compact. Then, there exists $\Sigma\subset X$ with
zero Lebesgue measure such that
$$
\omega(P)\subset E\quad\mbox{ for all }\quad P\in X\setminus \Sigma.
$$
\end{theorem}

\end{appendices}

\end{document}